\documentclass[12pt,reqno]{amsart}
\usepackage{color}
\usepackage[english]{babel}
\usepackage{amsfonts}
\usepackage{amsmath}
\usepackage{centernot}
\usepackage{amsthm}
\usepackage{amssymb}
\usepackage[misc]{ifsym}
\usepackage{bbm}
\usepackage{mathrsfs}
\usepackage{esint}
\usepackage{faktor}
\usepackage[utf8]{inputenc}
\usepackage{afterpage}
\usepackage[left=2.9cm,right=2.9cm,top=2.8cm,bottom=2.8cm]{geometry}
\usepackage{graphicx}
\usepackage[pagewise,mathlines]{lineno}
\usepackage{enumitem}
\usepackage[colorlinks=true,urlcolor=blue,citecolor=blue,linkcolor=blue,linktocpage,pdfpagelabels,bookmarksnumbered,bookmarksopen]{hyperref}
\usepackage[dvipsnames]{xcolor}
\newcommand{\ep}{\varepsilon}

\newcommand{\N}{\mathbb{N}}

\newcommand{\R}{\mathbb{R}}

\newcommand{\dx}{\, {\rm d} x}
\newcommand{\dy}{\, {\rm d} y}
\newcommand{\dt}{\, {\rm d} t}

\newcommand{\dtau}{\, {\rm d} \tau}

\newcommand{\eps}{\varepsilon}
\renewcommand{\phi}{\varphi}
\newtheorem{lemma}{Lemma}[section]
\newtheorem{thm}[lemma]{Theorem}
\newtheorem{prop}[lemma]{Proposition}

\theoremstyle{definition}
\newtheorem{defi}[lemma]{Definition}
\newtheorem{rmk}[lemma]{Remark}

\numberwithin{equation}{section} \DeclareMathOperator*{\esssup}{ess\, sup} 

 \DeclareMathOperator*{\dom}{dom}

\begin{document}\title[Non-local convective problems with singular nonlinearity]{Fractional Dirichlet problems with singular and non-locally convective reaction}
%
\author[L. Gambera]{Laura Gambera}
\address[L. Gambera]{Dipartimento di Matematica e Informatica, Universit\`a degli Studi di Palermo, Via Archirafi 34, 90123 Palermo, Italy}
\email{laura.gambera@unipa.it}
\author[S.A. Marano]{Salvatore A. Marano}
\address[S.A. Marano]{Dipartimento di Matematica e Informatica, Universit\`a degli Studi di Catania, Viale A. Doria 6, 95125 Catania, Italy}
\email{marano@dmi.unict.it}
\begin{abstract}
In this paper, the existence of positive weak solutions to a Dirichlet problem driven by the fractional $(p,q)$-Laplacian and with reaction both weakly singular and non-locally convective (i.e., depending on the distributional Riesz gradient of solutions) is established. Due to the nature of the right-hand side, we address the problem via sub-super solution methods, combined with variational techniques, truncation arguments, as well as fixed point results. 
\end{abstract}
\let\thefootnote\relax
%
%
\maketitle
\section{Introduction} 
Let $\Omega\subseteq\R^N$, $N\geq 2$, be a bounded domain with a $C^{1,1}$ boundary $\partial \Omega$, let $0<s_2\leq s\leq s_1\leq 1$, and let $2<q<p<\frac{N}{s_1}$ with $s_1p>1$.
%
Consider the problem 
\begin{equation}\label{prob}\tag{P}
\left\{
\begin{alignedat}{2}
(-\Delta)^{s_1}_p u+(-\Delta)^{s_2}_q u & = f(x,u)+g(x,D^s u)\;\; &&\mbox{in}\;\;\Omega,\\
u & >0 && \mbox{in}\;\;\Omega,\\
u & =0 && \mbox{in}\;\;\R^N\setminus\Omega,
\end{alignedat}
\right.
\end{equation}
where, given $r>1$ and $0<t<1$, $(-\Delta)_r^t$ denotes the (negative) fractional $r$-Laplacian, formally defined by
$$(-\Delta)^{t}_{r} u(x):=2\lim_{\ep\to 0^+}\int_{\R^N\setminus B_\ep(x)}
\frac{|u(x)-u(y)|^{r-2}(u(x)-u(y))}{|x-y|^{N+tr}}\dy,\quad x\in\R^N.$$
The symbol $D^s u$ indicates the \textit{distributional Riesz fractional gradient} of $u$ according to \cite{SS1,SS2}. If $u$ is sufficiently smooth and appropriately decays at infinity then
\begin{equation*}
D^s u(x):=c_{N,s}\lim_{\ep\to 0^+}\int_{\R^N\setminus B_\ep(x)}
\frac{u(x)-u(y)}{|x-y|^{N+s}}\frac{x-y}{|x-y|}\dy,\quad x\in\R^N,
\end{equation*}
with $c_{N,s}>0$; cf. \cite[p. 289]{SS2}. Moreover, $f:\Omega\times\R^+\to\R^+_0$ and $g:\Omega\times\R^N\to\R^+_0$ stand for Carathéodory's functions such that
\begin{equation}\label{hypf}
\tag{${\rm H_{f_1}}$}
\left\{
\begin{aligned}
&\liminf_{t\to 0^+} f(x,t)=:L> 0\quad\mbox{uniformly in $x\in\Omega$,}\\
&f(x,t)\le c_1 t^{-\gamma}+c_2 t^r \quad\forall\;(x,t)\in\Omega\times\R^+,\\
\end{aligned}
\right.
\end{equation}
\begin{equation}\label{hypf2}
\tag{${\rm H_{f_2}}$}
\text{$ t\mapsto\frac{f(\cdot,t)}{t^{q-1}}$ is strictly decreasing on $\R^+$},
\end{equation}
\begin{equation}\label{hypg}
\tag{${\rm H_g}$}
g(x,\xi)=c_3 (1+|\xi|^\zeta),\quad (x,\xi)\in\Omega\times\R^N,
\end{equation}
for suitable $\gamma\in (0,1),$ $c_i>0,$ $i=1,2,3$, $r,\zeta\in (1,p-1)$.

Since $s_2\leq s_1$, we are naturally led to solve problem \eqref{prob} in the fractional Sobolev space $W^{s_1,p}_0(\Omega)$. Precisely,
\begin{defi}
A function $u\in W^{s_1,p}_{0}(\Omega)$ is called a \textit{weak solution} of \eqref{prob} when $u>0$ a.e. in $\Omega$ and
\begin{equation*}
\begin{split}
& \int_{\R^{2N}}
\frac{|u(x)-u(y)|^{p-2}(u(x)-u(y))(\phi(x)-\phi(y))}{|x-y|^{N+s_1p}}\dx\dy\\
& \hskip1cm+\int_{\R^{2N}}
\frac{|u(x)-u(y)|^{q-2}(u(x)-u(y))(\phi(x)-\phi(y))}{|x-y|^{N+s_2q}}\dx\dy\\
& = \int_\Omega f(\cdot,u)\phi\dx+\int_\Omega g(\cdot,D^s u)\phi\dx\quad\forall\,
\phi\in W^{s_1,p}_0(\Omega).
\end{split}
\end{equation*}
\end{defi}
Let us next point out some hopefully newsworthy aspects, namely
\begin{itemize}
\item the driving differential operator is neither local nor homogeneous and no parameters appear on the right-hand side,
\item $f(x,\cdot)$ can be singular at zero, which means $\displaystyle{\lim_{t\to 0^+}} f(x,t)=+\infty$, and
\item the reaction is also non-locally convective, because $g$ depends on the distributional fractional gradient of solutions.
\end{itemize}
In latest years, the study of non-local differential equations has seen significant growth, mainly due to their many applications in real-world problems, such as game theory, finance, image processing, and materials science. A wealth of existence and uniqueness or multiplicity theorems are already available; see, e.g., \cite{FI,IPS,MM}. Several regularity results have also been published; let us mention \cite{IMS,IMS2} for the fractional $p$-Laplacian and \cite{G} as regards the fractional $(p,q)$-Laplacian. Finally, the survey \cite{DNPV} provides an exhaustive account on the corresponding functional framework, i.e., fractional Sobolev spaces. 

Dirichlet problems driven by non-local operators and with singular reactions were well investigated in \cite{CSM,G}. The work \cite{CSM} treats existence and uniqueness of solutions to the equation
$$(-\Delta)_p^su=\frac{a(x)}{u^\sigma}\;\;\mbox{in}\;\;\Omega,$$
where $\sigma>0$ and $a:\Omega\to\R^+$ fulfills appropriate conditions, while \cite{G} studies the more general situation
$$(-\Delta)^{s_1}_p u+(-\Delta)^{s_2}_q u=\frac{a_{\delta}(x)}{u^\sigma}\;\;\mbox{in} \;\;\Omega,$$
being $a_\delta\in L^{\infty}_{\text{loc}}(\Omega)$ a function that behaves like $d(x)^{-\delta}$,
%
%
with $d(x):={\rm dist}(x,\partial\Omega)$ and $\delta\in [0,s_1p)$. It should be noted that, contrary to \cite{CSM,G}, here, the reaction term $f(\cdot, u)$ is not necessarily a perturbation of $u^{-\sigma}$.

Finally, distributional fractional gradients were first introduced by Horváth \cite{H}, but gained traction and saw wider application especially after the two seminal papers of Shieh - Spector \cite{SS1, SS2}. Among recent contributions on this subject, we mention  \cite{SSS,CS1,B,CS2,BC}, as well as the references therein. The main reasons that contributed to spreading the use of Riesz gradients probably are: 1) Non-local versions of many classical results on Sobolev spaces can be obtained through them. 2) $D^s u$ formally tends to $\nabla u$ as $s\to 1^-$. 3) Significant geometrical and physical properties (invariance under translations or rotations, homogeneity of order $s$, etc.) remain true in this new context; cf. \cite{SI}. 

Although weakly singular (namely $\gamma<1$) problems are normally investigated via variational methods, here, the presence of $D^s u$ inside the reaction prevents this approach. That's way the existence of at least one positive solution to \eqref{prob} is established by means of sub-super solution arguments, variational and truncation techniques, besides Schauder's fixed point theorem. 

The paper is organized as follows. Preliminary facts are collected in Section \ref{prelim}. The next one shows that a suitable auxiliary problem, obtained by freezing the convection term, admits a unique solution, while \eqref{prob} is solved in Section \ref{mainresult}. 
\section{Preliminaries}\label{prelim}
Let $X$ be a real Banach space with topological dual $X^*$ and duality brackets $\langle\cdot,\cdot\rangle$. A function $A:X\to X^*$ is called:
\begin{itemize}
\item \emph{monotone} when $\langle A(x)-A(z),x-z\rangle\geq 0$ for all $x,z\in X$.
\item \emph{of type $(\mathrm{S})_+$} provided
\begin{equation*}
x_n\rightharpoonup x\;\;\mbox{in $X$,}\;\;\limsup_{n\to+\infty}\langle A(x_n),x_n-x\rangle \le 0\implies x_n\to x\;\;\mbox{in $X$.}   
\end{equation*}
\end{itemize}
The next elementary result will ensure that condition $(\mathrm{S})_+$ holds true for the fractional $(p,q)$-Laplacian.
\begin{prop}\label{sumop}
Let $A:X\to X^*$ be of type $(\mathrm{S})_+$ and let $B:X\to X^*$ be monotone. Then $A+B$ satisfies condition $(\mathrm{S})_+$. 
\end{prop}
\begin{proof}
Suppose $x_n\rightharpoonup x$ in $X$ and 
\begin{equation}\label{Splus}
\limsup_{n\to+\infty}\langle A(x_n)+B(x_n),x_n-x\rangle\le 0.   
\end{equation}
The monotonicity of $B$ entails
\begin{equation*}
\begin{split}
\langle A(x_n), & x_n-x\rangle=\langle A(x_n)+B(x_n),x_n-x\rangle-\langle B(x_n),
x_n-x\rangle\\
& =\langle A(x_n)+B(x_n),x_n-x\rangle-\langle B(x_n)-B(x),x_n-x\rangle-\langle B(x),
x_n-x\rangle\\
&\leq \langle A(x_n)+B(x_n),x_n-x\rangle-\langle B(x),x_n-x\rangle\quad\forall\, n\in\N.\\
\end{split}    
\end{equation*}
Using \eqref{Splus} we thus get
$$\limsup_{n\to+\infty}\langle A(x_n),x_n-x\rangle\leq 0,$$
whence $x_n\to x$ because $A$ is of type $(\mathrm{S})_+$. 
\end{proof}
Finally, if $X$ and $Y$ are two topological spaces then $X\hookrightarrow Y$ means that $X$ continuously embeds in $Y$.

Hereafter, $\Omega$ is a bounded domain of the real Euclidean $N$-space $(\mathbb{R}^N, |\cdot|)$, $N\geq 2$, with a $C^2$-boundary $\partial\Omega$, $|E|$ indicates the $N$-dimensional Lebesgue measure of $E\subseteq \mathbb{R}^{N}$, 
$$t_\pm:=\max\{\pm t,0\},\quad t\in\R,$$
while $C$, $C_1$, etc. are positive constants, which may change value from line to line, whose dependencies will be specified when necessary. Denote by $d:\overline{\Omega}\to \R^+_0$ the distance function of $\Omega$, i.e., 
$$d(x):={\rm dist}(x,\partial\Omega)\quad\forall\, x\in\overline{\Omega}.$$
It enjoys a useful summability property (see \cite[Proposition 2.1]{GM}), namely
\begin{prop}\label{distsumm}
If $0<\sigma<1<q<\frac{1}{\sigma}$ then $d^{-\sigma}\in L^q(\Omega)$.
\end{prop}
Let $X(\Omega)$ be a real-valued function space on $\Omega$ and let $u,v\in X(\Omega)$. We simply write $u\leq v$ when $u(x)\leq v(x)$ a.e. in $\Omega$. Analogously for $u<v$, etc. To shorten notation, define
\begin{equation*}
\Omega (u\leq v):=\{x\in \Omega :u(x)\leq v(x)\},\quad X(\Omega)_+:=\{w\in X(\Omega): w>0\}.
\end{equation*}
Henceforth, $p'$ indicates the conjugate exponent of $p\geq 1$, the Sobolev space $W^{1,p}_0(\Omega)$ is equipped with Poincaré's norm 
\begin{equation*}
\Vert u\Vert_{1,p}:=\Vert |\nabla u|\Vert_p,\quad u\in W^{1,p}_0(\Omega),
\end{equation*}
where, as usual,
\begin{equation*}
\Vert v\Vert_q:=\left\{ 
\begin{array}{ll}
\left(\int_{\Omega }|v(x)|^q\dx\right)^{1/q} & \text{ if }1\leq q<+\infty, \\ 
\phantom{} &  \\ 
\underset{x\in\Omega}{\esssup}\, |v(x)| & \text{ when } q=+\infty,
\end{array}
\right.
\end{equation*}
and, given any $u\in W^{1,p}_0(\Omega)$, we set $u:=0$ a.e. in $\R^N\setminus\Omega$; cf. \cite[Section 5]{DNPV}. Moreover, $W^{-1,p'}(\Omega):=(W^{1,p}_0(\Omega))^*$ while $p^*$ is the Sobolev critical exponent for the embedding $W^{1,p}_0(\Omega) \hookrightarrow L^q(\Omega)$. It is known that $p^*=\frac{Np}{N-p}$ once $p<N$. 

Fix $s\in(0,1)$. The Gagliardo semi-norm of a measurable function $u:\R^N\to\R$ is
\begin{equation*}
[u]_{s,p}:=\left(\int_{R^N\times\R^N}\frac{|u(x)-u(y)|^p}{|x-y|^{N+ps}}{\dx}{\dy} \right)^{1/p}.
\end{equation*}
$W^{s,p}(\R^N)$ denotes the fractional Sobolev space
\begin{equation*}
W^{s,p}(\R^N):= \left\{u\in L^p(\R^N):\ [u]_{s,p}<+\infty \right\}
\end{equation*}
endowed with the norm
\begin{equation*}
\|u\|_{W^{s,p}(\mathbb{R}^N)}:=(\|u\|^p_{L^p(\mathbb{R}^N)}+[u]_{s,p}^p)^{1/p}.
\end{equation*}
On the space
\begin{equation*}
W^{s,p}_0(\Omega):=\{u\in W^{s,p}(\R^N):u=0\;\;\mbox{a.e. in}\;\;\R^N\setminus\Omega\}
\end{equation*}
we will consider the equivalent norm
$$\| u\|_{s,p}:=[u]_{s,p},\quad u\in W^{s,p}_0(\Omega).$$
As before, $W^{-s,p'}(\Omega):=(W^{s,p}_0(\Omega))^*$ and $p^*_s$ indicates the fractional Sobolev critical exponent, i.e., $p^*_s=\frac{Np}{N-sp}$ when $sp<N$, $p^*_s= +\infty$ otherwise. Thanks to Propositions 2.1--2.2, Theorem 6.7, and Corollary 7.2 of \cite{DNPV} one has
\begin{prop}\label{fracemb}
If $1\leq p<+\infty$ then:
\begin{itemize}
\item[{\rm (a)}] $0<s'\le s''\le 1\;\implies\; W^{s'',p}_0(\Omega)\hookrightarrow W^{s',p}_0(\Omega)$. 
\item[{\rm (b)}] $W^{s,p}_0(\Omega)\hookrightarrow L^q(\Omega)$ for all $q\in [1, p^*_s]$.
\item[{\rm (c)}] The embedding in {\rm (b)} is compact once $q<p^*_s<+\infty$.
\end{itemize}
\end{prop}
However, contrary to the non-fractional case, 
$$1\leq q<p\leq+\infty\;\;\centernot\implies W^{s,p}_0(\Omega) \subseteq W^{s,q}_0(\Omega);$$
cf. \cite{MS}. Define, for every $u,v\in W^{s,p}_0(\Omega)$,
\begin{equation*}
\langle(-\Delta)^s_p u,v\rangle:=
\int_{\R^N\times\R^N}\frac{|u(x)-u(y)|^{p-2}(u(x)-u(y))(v(x)-v(y))}{|x-y|^{N+sp}}\dx\dy\, .
\end{equation*}
The operator $(-\Delta)_p^s$ is called (negative) $s$-fractional $p$-Laplacian. It possesses the following properties.
\begin{itemize}
\item[$({\rm p}_1)$] $(-\Delta)^s_p:W^{s,p}_0(\Omega)\rightarrow W^{-s,p'}(\Omega)$ is monotone, continuous, and of type $({\rm S})_+$; vide \cite[Lemma 2.1]{FI}.
\item[$({\rm p}_2)$] $(-\Delta)^s_p$ maps bounded sets into bounded sets. In fact,
$$\Vert(-\Delta)^s_p u\Vert_{W^{-s,p'}(\Omega)}\leq\Vert u\Vert_{s,p}^{p-1} 
\quad\forall\, u\in W^{s,p}_0(\Omega).$$
\end{itemize}
To deal with distributional fractional gradients, we first introduce the Bessel potential spaces $L^{\alpha,p}(\R^N)$, where $\alpha>0$. Set, for every $x\in\R^N$,
\begin{equation*}
g_\alpha(x):= \frac{1}{(4 \pi)^{\alpha/2}\Gamma (\alpha/2)} \int_{0}^{+\infty}e^{\frac{-\pi|x|^2}{\delta}} e^{\frac{-\delta}{4\pi}}\delta^{\frac{\alpha-N}{2}}\frac{{\rm d}\delta}{\delta}\, .  
\end{equation*}
On account of \cite[Section 7.1]{MI} one can assert that:
\begin{itemize}
\item[1)] $g_\alpha\in L^1(\R^N)$ and $\|g_\alpha\|_{L^1(\R^N)}=1$.
\item[2)] $g_\alpha$ enjoys the semigroup property, i.e., $g_\alpha \ast g_\beta= g_{\alpha+\beta}$ for any $\alpha,\beta>0$.
\end{itemize}
Now, put
\begin{equation*}
L^{\alpha,p}(\R^N):=\{u:\, u= g_\alpha\ast\tilde u\;\mbox{for some}\;\tilde u\in L^p(\R^N)\}
\end{equation*}
as well as
$$\|u\|_{L^{\alpha,p}(\R^N)}= \|\tilde u\|_{ L^p(\R^N)}\;\;\mbox{whenever}\;\; u=g_\alpha\ast\tilde u.$$ 
Using 1)--2) easily yields
$$0<\alpha<\beta\;\implies\; L^{\beta,p}(\R^N)\subseteq L^{\alpha,p}(\R^N)\subseteq L^p(\R^N).$$
Moreover (see \cite[Theorem 2.2]{SS1}),
%
%
\begin{thm}\label{besselspace} 
If $1<p<+\infty$ and $0<\eps<\alpha$ then
\begin{equation*}
L^{\alpha+\eps,p}(\R^N)\hookrightarrow W^{\alpha,p}(\R^N)\hookrightarrow L^{\alpha-\eps,p}(\R^N).
\end{equation*}
\end{thm}
%
%
Finally, define
\begin{equation*}
 L^{s,p}_0(\Omega):=\{u\in L^{s,p}(\R^N): u=0\;\text{in}\;\R^N\setminus\Omega \}.
\end{equation*}
Thanks to Theorem \ref{besselspace} we clearly have
\begin{equation}\label{comparisonbound}
L^{s+\eps,p}_0(\Omega)\hookrightarrow W^{s,p}_0(\Omega)
\hookrightarrow L^{s-\eps,p}_0(\Omega)\quad\forall\,\eps\in (0,s).
\end{equation}
The next basic notion is taken from \cite{SS1}. For $0<\alpha<N$, let
\begin{equation*}
\gamma(N,\alpha):=\frac{\Gamma((N-\alpha)/2)}{\pi^{N/2}2^\alpha\Gamma(\alpha/2)},
\quad I_\alpha(x):=\frac{\gamma(N,\alpha)}{|x|^{N-\alpha}},\quad x\in\R^N\setminus\{0\}.
\end{equation*}
If $u\in L^p(\R^N)$ and $I_{1-s}\ast u$ makes sense then the vector
$$D^s u:=\left(\frac{\partial}{\partial x_1}(I_{1-s}\ast u),\ldots,
\frac{\partial}{\partial x_N}(I_{1-s}\ast u)\right),$$
where partial derivatives are understood in a distributional sense, is called distributional Riesz $s$-fractional gradient of $u$. Theorem 1.2 in \cite{SS1} ensures that
$$D^s u=I_{1-s}\ast Du\quad\forall\, u\in C^\infty_c(\R^N).$$
Further, $D^s u$ looks like the natural extension of $\nabla u$ to the fractional framework; cf., e.g., \cite{GMM2} for details. According to \cite[Definition 1.5]{SS1}, $X^{s,p}(\R^N)$ denotes the completion of $C^\infty_c(\R^N)$ with respect to the norm
$$\|u\|_{X^{s,p}(\R^N)}:=(\|u\|_{L^p(\R^N)}^p+\|D^s u\|_{L^p(\R^N)}^p)^{1/p}.$$
Since, by \cite[Theorem 1.7]{SS1}, $X^{s,p}(\R^N)=L^{s,p}(\R^N)$ we can deduce
many facts about $X^{s,p}(\R^N)$ from the existing literature on $L^{s,p}(\R^N)$. In particular, if
\begin{equation*}
X^{s,p}_{0}(\Omega):=\{u \in X^{s,p}(\R^N):u=0\;\text{in}\;\R^N\setminus\Omega\}.    
\end{equation*}
then $X^{s,p}_{0}(\Omega)=L^{s,p}_{0}(\Omega)$.  
\section{freezing the convection term}
To address the two troubles (singularity and convection) separately, here, we will study an auxiliary equation
%
%
patterned after that of \eqref{prob}, but with $D^s u$ replaced by $D^s v$ for fixed $v\in W^{s_1,p}_0(\Omega)$.
\begin{lemma}\label{torsionproblem}
Under hypothesis \eqref{hypf}, the problem 
\begin{equation}\label{probf}
\tag{${\rm P}_f$}
\left\{
\begin{alignedat}{2}
(-\Delta)^{s_1}_{p} u+(-\Delta)^{s_2}_{q}u & =f(x,u) &&\quad\mbox{in}\;\;\Omega,\\
u & =0 &&\quad \mbox{in}\;\;\R^N\setminus\Omega,
\end{alignedat}
\right.
\end{equation}
possesses a positive sub-solution $\underline{u}\in W^{s_1,p}_0(\Omega)\cap C^{0,\tau}(\overline{\Omega})$, where $\tau\in (0, s_1)$.
\end{lemma}
\begin{proof}
Thanks to \eqref{hypf}, for every $\eps\in (0,L)$ there exists $\delta\in (0,1)$ such that 
$$f(x, t)>\eps\quad\forall\, (x, t)\in\Omega\times (0, \delta).$$
Let $\sigma>0$. Theorem 3.15 of \cite{G} provides a positive solution $u_\sigma \in W^{s_1,p}_{0}(\Omega)\cap C^{0,\tau}(\overline{\Omega})$, $\tau\in (0, s_1)$, of the torsion problem
\begin{equation*}
\left\{
\begin{alignedat}{2}
(-\Delta)^{s_1}_p u+(-\Delta)^{s_2}_q u & =\sigma &&\quad\mbox{in}\;\;\Omega,\\
u & =0 &&\quad \mbox{in}\;\;\R^N\setminus\Omega.
\end{alignedat}
\right.
\end{equation*}
Moreover, $u_\sigma\to 0$ in $C^{0,\tau}(\overline{\Omega})$ as $\sigma\to 0^+$.
Thus, for any $\sigma$ sufficiently small one has both $\sigma<\ep$ and $\|u_\sigma\|_{\infty}<\delta$. This evidently implies
$$(-\Delta)^{s_1}_p u_\sigma+(-\Delta)^{s_2}_q u_\sigma=\sigma< \eps<
f(\cdot,u_\sigma),$$ 
i.e., $\underline{u}:=u_\sigma$ is a positive sub-solution of \eqref{probf}.
\end{proof}
\begin{rmk}\label{usin}
If $s_1\not= q's_2$ then Hopf's theorem \cite[Proposition 2.12]{G} ensures that
\begin{equation}\label{hopf}
\eta d(x)^{s_1}\leq\underline{u}(x)\quad\forall\, x\in\Omega, 
\end{equation}
with suitable $\eta >0$. Otherwise, $\eta d^{\alpha}\leq\underline{u}$, being $\alpha>s_1$, $\alpha\not= q's_2$, and $\alpha\not=p's_1$; cf. \cite[ Remark 2.14]{G}. 
\end{rmk}
Now, fixed any $v\in W^{s_1,p}_{0}(\Omega)$, consider the following problem, where the convective term has been frozen:
\begin{equation}\label{auxprob}
\tag{${\rm {P_{v}}}$}
\left\{
\begin{alignedat}{2}
(-\Delta)^{s_1}_{p} u+(-\Delta)^{s_2}_{q}u & = f (x,u)+g(x, D^s v) && \quad\mbox{in}\;\;\Omega,\\
u & =0 &&\quad \mbox{on}\;\;\R^N\setminus\Omega.
\end{alignedat}
\right.
\end{equation}
\begin{thm}\label{existence}
Let \eqref{hypf} and \eqref{hypg} be satisfied. If $v\in W^{s_1,p}_{0}(\Omega)$ then \eqref{auxprob} admits a weak solution $u_v \in W^{s_1,p}_0(\Omega)\cap C^{0,\tau}(\overline{\Omega})$, where $\tau\in (0,s_1)$. Moreover, $u_v\ge \underline{u}$.
\end{thm}
\begin{proof}
Recalling Lemma \ref{torsionproblem}, define
\begin{equation}\label{deftildef}
\tilde{f}(x,t):=f(x,\max\{\underline{u}(x),t\}),\quad (x,t)\in\Omega\times\R.     
\end{equation}
Without loss of generality we can suppose $s_1\not=q's_2$. In fact, by Remark \ref{usin}, the case $s_1=q's_2$ is entirely analogous. Thanks to \eqref{hypf} and \eqref{hopf} one has
\begin{equation}\label{tildef}
\begin{split}
\tilde{f}(\cdot, t) &
\le c_1(\max\{\underline{u},t\})^{-\gamma}+c_2(\max\{\underline{u},t\})^{r}
\le c_1\underline{u}^{-\gamma}+c_2(\underline{u}^r+t^r)\\
&\le c_1(\eta d^{s_1})^{-\gamma}+c_2 (\max_{\overline{\Omega}}\underline{u})^r+ c_2 t^r
\le C_1(d^{-\gamma s_1}+t^r+1),\;\; t\in\R^+.
\end{split}
\end{equation}
The energy functional $\tilde{J}:W^{s_1,p}_0(\Omega)\to\R$ associated with the problem
\begin{equation}\label{frizedprob}
\left\{
\begin{alignedat}{2}
(-\Delta)^{s_1}_p u+(-\Delta)^{s_2}_q u & = \tilde f(x,u)+g(x,D^s v) && \quad \mbox{in}\;\;\Omega,\\
u & =0 && \quad\mbox{in}\;\;\R^N\setminus\Omega,
\end{alignedat}
\right.
\end{equation}
is written as
\begin{equation*}
\begin{split}
\tilde{J}(u):=
\frac{1}{p} & \int_{\R^N\times\R^N}\frac{|u(x)-u(y)|^{p}}{|x-y|^{N+s_1p}}\dx\dy+ \frac{1}{q}\int_{\R^N\times\R^N}\frac{|u(x)-u(y)|^{q}}{|x-y|^{N+s_2q}}\dx\dy\\
- & \int_{\Omega}\tilde{F}(\cdot,u)\dx-\int_{\Omega} G(\cdot,D^sv)\dx,
\quad u\in W^{s_1,p}_0(\Omega),
\end{split} 
\end{equation*}
where
$$\tilde{F}(x,\tau):=\int_0^\tau\tilde{f}(x,t)\dt,\qquad 
G(x,\xi):=\int_0^\tau g(x,\xi)\dt =\tau g(x, \xi).$$
Obviously, $\tilde{J}$ turns out well defined and of class $C^1$. Moreover, \eqref{tildef}
easily entails
\begin{equation}\label{new}
\tilde{F}(x,\tau)\leq\int_0^{|\tau|}\tilde{f}(x,t)\dt\le
C_1\left[(d^{-\gamma s_1}+1)|\tau|+\frac{|\tau|^r}{r+1}\right]\quad\forall\, \tau\in\R,
\end{equation}
because $\tilde{f}(x,t)\geq 0$. From \eqref{new}, \eqref{hypg}, H\"older's inequality, and fractional Hardy’s inequality \cite[Theorem 1.1]{Dy} (recall that $s_1p>1$) it follows 
\begin{equation*}
\begin{split}
\tilde{J}(u) & \ge\frac{1}{p}
\int_{\R^N\times\R^N}\frac{|u(x)-u(y)|^{p}}{|x-y|^{N+s_1p}}\dx\dy
-C_1\int_\Omega d^{-\gamma s_1}|u|\dx
-\frac{C_1}{r+1}\int_\Omega |u|^{r+1} \dx\\
& -c_3\int_\Omega |D^s v|^\zeta |u|\dx-(C_1+c_3)\int_\Omega |u|\dx\\
& \ge\frac{1}{p}\int_{\R^N\times\R^N}\frac{|u(x)-u(y)|^{p}}{|x-y|^{N+s_1p}}\dx\dy
-C_1\int_\Omega d^{(p-\gamma)s_1}\frac{|u|}{d^{s_1p}}\dx
-\frac{C_1}{r+1}\int_\Omega |u|^{r+1} \dx\\
& -c_3\int_\Omega |D^s v|^\zeta |u|\dx-C_2\|u\|_p\\
& \ge\frac{1}{p}\|u\|^p_{s_1,p}-C_3(\|u\|_{s_1,p}+\|u\|_p^{r+1}+\|D^sv\|_p^\zeta\|u\|_{\frac{p}{p-\zeta}}+\|u\|_p),\quad
u\in W^{s_1,p}_0(\Omega).
\end{split} 
\end{equation*}
Since $r,\zeta\in (1,p-1)$, through Proposition \ref{fracemb} (b) we see that $\tilde{J}$ is coercive. Thus, by Weierstrass-Tonelli's theorem, there exists $u_v\in W^{s_1,p}_{0}(\Omega)$ fulfilling
$$\tilde{J}(u_v) =\inf_{u\in W^{s_1,p}_{0}(\Omega)}\tilde{J}(u),$$
whence $u_v$ turns out a weak solution to \eqref{frizedprob}. As in the proof of \cite[Proposition 2.10]{IMS} one has $(\underline{u}-u_v)_+\in W^{s_1,p}_0(\Omega)$. Testing \eqref{frizedprob} with $\phi:=(\underline{u}-u_v)_+$ yields
\begin{equation*}
\langle(-\Delta)^{s_1}_p u_v+(-\Delta)^{s_2}_q u_v,\phi\rangle
=\int_\Omega\tilde{f}(\cdot,u_v)\phi\dx+\int_\Omega g(\cdot,D^s v)\phi\dx.   
\end{equation*}
Lemma \ref{torsionproblem} and the inequality $g(x,\xi)\geq0$ produce
\begin{equation*}
\langle(-\Delta)^{s_1}_p\underline{u}+(-\Delta)^{s_2}_q\underline{u},\phi\rangle \leq\int_\Omega f(\cdot,\underline{u})\phi\dx
\le\int_\Omega f(\cdot,\underline{u})\phi\dx+\int_\Omega g(\cdot,D^s v)\phi\dx.  
\end{equation*}
Therefore, by $({\rm p}_1)$ and \eqref{deftildef}, 
\begin{equation*}
\begin{split}
\langle(-\Delta)^{s_1}_p \underline{u}-(-\Delta)^{s_1}_p u_v,\phi\rangle & \\
& \le\langle(-\Delta)^{s_1}_p \underline{u}-(-\Delta)^{s_1}_p u_v,\phi\rangle 
+\langle(-\Delta)^{s_2}_q\underline{u}-(-\Delta)^{s_2}_{q} u_v,\phi\rangle\\
&\le\int_{\Omega(u_v<\underline{u})}(f(\cdot,\underline{u})-\tilde{f}(\cdot,u_v)) (\underline{u}-u_v)\dx\\
& = \int_{\Omega(u_v<\underline{u})}(f(\cdot,\underline{u})-f(\cdot,\underline{u}))(\underline{u}-u_v)\dx=0.
\end{split} 
\end{equation*}
Now, Lemma 9 in \cite{LL} forces
\begin{equation*}
0<\underline{u}\le u_v.
\end{equation*} 
Consequently, $u_v\in W^{s_1,p}_0(\Omega)_+$ weakly solves \eqref{auxprob}. Corollary 2.10 of \cite{G} then ensures that $u_v\in C^{0,\tau}(\overline{\Omega})$ for all $\tau\in (0, s_1)$.
\end{proof}
\begin{lemma}\label{hidden convexity}
If $0<s<1$ while $\Phi:W^{s,p}_0(\Omega)\to\R^+_0$ is defined by
\begin{equation*}
\Phi(u):=\frac{1}{p}\int_{\R^{2N}}\frac{|u(x)-u(y)|^p}{|x-y|^{N+sp}}\dx\dy\quad\forall \, u\in W^{s,p}_0(\Omega) 
\end{equation*}
then the operator
\begin{equation*}
\hat\Phi(w):=\begin{cases}
\Phi\left(w^{1/q}\right) & \text{when $w\ge 0$ and $w^{1/q}\in W^{s,p}_0 (\Omega)$},\\
+\infty & \text{otherwise,}
\end{cases}
\end{equation*}
has a nonempty domain and is convex.
\end{lemma}
\begin{proof}
Pick $l>q$ and a non-negative $u\in W^{1,p}_0(\Omega)\cap L^\infty(\Omega)$. One has $u^{l/q}\in W^{1,p}_0(\Omega)$, because
$$\int_\Omega |\nabla(u^{l/q})|^p\dx
=\int_\Omega\left(\frac{l}{q} u^{l/q-1}|\nabla u|\right)^p\dx
\le C\int_\Omega|\nabla u|^p \dx<+\infty.$$
So, $u^{l/q}\in W^{s,p}_0(\Omega)$ by Proposition \ref{fracemb}. Here, as usual, $u\equiv 0$ on $\R^N\setminus\Omega$. We claim that $\hat{\Phi}(u^l)<+\infty$. In fact,
\begin{equation*}
\begin{split}
\hat{\Phi} & (u^l)=\Phi(u^{l/q})
=\frac{1}{p}\int_{\R^{2N}}\frac{|u(x)^{l/q}-u(y)^{l/q}|^p}{|x-y|^{N+sp}}\dx\dy\\
& =\frac{1}{p}
\int_{\R^N}\dx\int_{B_1(x)}\frac{|u(x)^{l/q}-u(y)^{l/q}|^p}{|x-y|^{N+sp}}\dy
+\frac{1}{p}
\int_{\R^N}\dx\int_{\R^N\setminus B_1(x)}
\frac{|u(x)^{l/q}-u(y)^{l/q}|^p}{|x-y|^{N+sp}}\dy.
\end{split}  
\end{equation*}
Moreover,
\begin{equation*}
\begin{split}
& \int_{\R^N}\dx\int_{B_1(x)}\frac{|u(x)^{l/q}-u(y)^{l/q}|^p}{|x-y|^{N+sp}}\dy
=\int_{\R^N}\dx\int_{B_1(0)}\frac{|u(x)^{l/q}-u(x+y)^{l/q}|^p}{|y|^{N+sp}}\dy\\
& = \int_{\R^N}\dx\int_{B_1(0)}\frac{|u(x)^{l/q}-u(x+y)^{l/q}|^p}{|y|^p}
\frac{1}{|y|^{N+(s-1)p}}\dy\\
& \le\int_{\R^N}\dx\int_{B_1(0)}\left(\int_0^1 |\nabla u(x+\tau y)^{l/q}|\dtau \right)^p \frac{1}{|y|^{N+(s-1)p}}\dy\\
& \le C\int_{\R^N}\int_{B_1(0)}\int_0^1
\frac{|\nabla u(x+\tau y)^{l/q}|^p}{|y|^{N+(s-1)p}}\dx\dy\dtau\\
& \le C\int_{B_1(0)}\int_0^1\frac{\|\nabla u^{l/q}\|^p_{L^p(\R^N)}}{|z|^{N+(s-1)p}}\dy\dtau\le C_1\|\nabla u^{l/q}\|^p_{L^p(\R^N)}<+\infty.
\end{split}
\end{equation*}
Likewise,
\begin{equation*}
\begin{split}
\int_{\R^N}\dx & \int_{\R^N\setminus B_1(x)}
\frac{|u(x)^{l/q}-u(y)^{l/q}|^p}{|x-y|^{N+sp}}\dy\\
&\le 2^{p-1}\int_{\R^N}\dx\int_{\R^N\setminus B_1(x)}
\frac{|u(x)^{l/q}|^p+|u(y)^{l/q}|^p}{|x-y|^{N+sp}}\dy\\
& = 2^{p-1}\int_{\R^N}\dx\int_{\R^N\setminus B_1(0)}
\frac{|u(x)^{l/q}|^p+|u(x+y)^{l/q}|^p}{|y|^{N+sp}}\dy\\
& \le C_3\int_{\R^N} |u(x)^{l/q}|^p\dx=C_3\Vert u^{l/q}\Vert_p^p<+\infty
\end{split}
\end{equation*}
because $u^{l/q}\in W^{1,p}_0(\Omega)$. Hence, $u^l\in\dom\hat{\Phi}$, and the first conclusion follows. Next, let $u_1, u_2\in\dom\hat{\Phi}$ and let $t\in (0,1)$. If $v_i:= u_i^{1/q}$, $i=1,2$, and 
$$v_3:=((1-t)u_1+tu_2)^{1/q}$$
then, thanks to discrete hidden convexity \cite[Proposition 4.1]{BF},
\begin{equation*}
|v_3(x)-v_3(y)|^p\le (1-t)|v_1(x)-v_1(y)|^p+t |v_2(x)-v_2(y)|^p\quad\forall\, x, y \in \R^N.
\end{equation*}
This entails
\begin{equation*}\begin{split}
\hat\Phi((1-t)u_1+tu_2) & =\frac{1}{p}\int_{\R^{2N}}\frac{|v_3(x)-v_3(y)|^p}{|x-y|^{N+sp}}\dx\dy\\
& \le\frac{1-t}{p}\int_{\R^{2N}}\frac{|v_1(x)-v_1(y)|^p}{|x-y|^{N+sp}}\dx\dy
+\frac{t}{p}\int_{\R^{2N}}\frac{|v_2(x)-v_2(y)|^r}{|x-y|^{N+sp}}\dx\dy\\
& = (1-t)\hat\Phi (u_1)+t\hat\Phi (u_2),
\end{split}
\end{equation*}
thus completing the proof.
\end{proof}
\begin{rmk}
The above result holds true even when $q=p$, with the same proof.  
\end{rmk}
\begin{thm}
Under \eqref{hypf}--\eqref{hypf2} and \eqref{hypg}, for every fixed $v\in W^{s_1,p}_0(\Omega)$, the solution $u_v\in W^{s_1,p}_0(\Omega)\cap C^{0,\tau}(\overline{\Omega})$ to problem \eqref{auxprob} given by Theorem \ref{existence} is unique.
\end{thm}
\begin{proof}
Suppose $u_v, w_v\in W^{s_1,p}_0(\Omega)\cap C^{0,\tau}(\overline{\Omega})$ solve \eqref{auxprob}, namely
\begin{equation}\label{uv}
\langle(-\Delta)^{s_1}_p u_v+(-\Delta)^{s_2}_q u_v,\phi\rangle=
\int_\Omega f(\cdot,u_v)\phi\dx+\int_\Omega g(\cdot,D^s v)\phi\dx,
\end{equation}
\begin{equation}\label{wv}
\langle(-\Delta)^{s_1}_p w_v+(-\Delta)^{s_2}_q w_v,\psi\rangle=
\int_\Omega f(\cdot,w_v)\psi\dx+\int_\Omega g(\cdot,D^s v)\psi\dx
\end{equation}
for all $\phi,\psi\in W^{s_1, p}_0(\Omega)$. The functions 
$$\phi:=\frac{{u_v}^q-{w_v}^q}{{u_v}^{q-1}}\quad\mbox{and}\quad
\psi:=\frac{{u_v}^q-{w_v}^q}{w_v^{q-1}}$$
lie in $W^{s_1, p}_0(\Omega)$, because $u_v,w_v\in C^{0,\tau}(\overline{\Omega})_+$. Hence, via \eqref{uv}--\eqref{wv} we achieve
\begin{equation}\label{subtraction}
\begin{split}
\langle( & -\Delta)^{s_1}_p u_v,\phi\rangle-\langle(-\Delta)^{s_1}_p w_v,\psi\rangle
+\langle(-\Delta)^{s_2}_{q}u_v,\phi\rangle-\langle(-\Delta)^{s_2}_{q}w_v,\psi\rangle\\
& =\int_\Omega\left(\frac{f(\cdot,u_v)}{u_v^{q-1}}
-\frac{f(\cdot,w_v)}{w_v^{q-1}}\right) (u_v^q-w_v^q)\dx
+\int_\Omega g(\cdot,D^s v) (\phi-\psi)\dx.
\end{split}    
\end{equation}
Lemma \ref{hidden convexity} ensures that the functional $\hat{J}$ associated with
\begin{equation*}
J(u):=\frac{1}{p}\int_{\R^{2N}}\frac{|u(x)-u(y)|^p}{|x-y|^{N+s_1p}}\dx\dy,\quad u\in W^{s_1,p}_0(\Omega),
\end{equation*}
turns out convex. Therefore, after a standard computation,
\begin{equation}\label{convexityp}
0\le q\langle \hat{J}'(u_v^q)-\hat{ J}'(w_v^q),u_v^q-w_v^q\rangle
=\langle(-\Delta)^{s_1}_p u_v,\phi\rangle-\langle(-\Delta)^{s_1}_p w_v,\psi\rangle.
\end{equation}
An analogous argument produces 
\begin{equation}\label{convexityq}
\langle(-\Delta)^{s_2}_q u_v,\phi\rangle-\langle(-\Delta)^{s_2}_q w_v,\psi\rangle\ge 0.
\end{equation}
Now, gathering \eqref{subtraction}--\eqref{convexityq} together yields
\begin{equation}\label{nonneg}
\int_\Omega\left(\frac{f(\cdot,u_v)}{u_v^{q-1}}-\frac{f(\cdot,w_v)}{w_v^{q-1}}\right)
(u_v^q-w_v^q)\dx+\int_\Omega g(\cdot,D^s v) (\phi-\psi)\dx\geq 0.    
\end{equation}
By \eqref{hypf2} the function $t\mapsto\frac{f(\cdot,t)}{t^{q-1}}$ is decreasing on $\R^+$. This implies
\begin{equation}\label{negone}
\int_\Omega\left(\frac{f(\cdot,u_v)}{u_v^{q-1}}-\frac{f(\cdot,w_v)}{w_v^{q-1}}\right)
(u_v^q-w_v^q)\dx\leq 0.
\end{equation}
Moreover,
\begin{equation}\label{negtwo}
\begin{split}
\int_\Omega g(\cdot,D^s v)(\phi-\psi)\dx & \leq
\int_{\Omega(u_v\ge w_v)} g(\cdot,D^s v)\left(\frac{u_v^q-w_v^q}{w_v^{q-1}}-\frac{u_v^q-w_v^q}{w_v^{q-1}}\right )\dx\\
& \hskip1cm +\int_{\Omega(u_v<w_v)} g(\cdot,D^s v)\left (\frac{u_v^q-w_v^q}{u_v^{q-1}}-\frac{u_v^q-w_v^q}{w_v^{q-1}}\right )\dx\\
& \le\int_{\Omega(u_v<w_v)} g(\cdot,D^s v)\left(-\frac{w_v^q-u_v^q}{u_v^{q-1}} +\frac{w_v^q-u_v^q}{u_v^{q-1}}\right)\dx=0.
\end{split}    
\end{equation}
From \eqref{nonneg}--\eqref{negtwo} it finally follows
\begin{equation*}
\int_\Omega\left(\frac{f(\cdot,u_v)}{u_v^{q-1}}-\frac{f(\cdot,w_v)}{w_v^{q-1}}\right)
(u_v^q-w_v^q)\dx=0,    
\end{equation*}
whence, due to \eqref{hypf2} again, $u_v\equiv w_v,$ as desired.
\end{proof}
\section{Main result}\label{mainresult}
Define, for every $v\in W^{s_1,p}_0 (\Omega)$,  
\begin{equation}\label{definitionT}
T(v):= u_v,
\end{equation}
$u_v\in W^{s_1,p}_0 (\Omega)_+$ being the unique solution of \eqref{auxprob} found in Theorem \ref{existence}.
\begin{lemma}\label{Schau}
Let \eqref{hypf}, \eqref{hypf2}, \eqref{hypg} be satisfied and let $q's_2\neq s_1< \frac{1}{p'\gamma}$. Then $T$ possesses a fixed point $u\in W^{s_1,p}_0(\Omega)$.
\end{lemma}
\begin{proof}
Given any $v\in W^{s_1,p}_0(\Omega)$, test problem \eqref{auxprob} with its solution $u_v$. Through \eqref{hypf} and \eqref{hypg} we thus arrive at
\begin{equation*}
\begin{split}
\|u_v\|^p_{s_1,p} & \le\int_{\R^{2N}}\frac{|u_v(x)-u_v(y)|^{p}}{|x-y|^{N+s_1p}}\dx\dy +\int_{\R^{2N}}\frac{|u_v(x)-u_v(y)|^{q}}{|x-y|^{N+s_2q}}\dx\dy\\
& \le c_1\int_\Omega u_v^{1-\gamma}\dx+c_2\int_\Omega u_v^{r+1}\dx
+c_3\int_\Omega(u_v+|D^s v|^\zeta u_v)\dx. 
\end{split}
\end{equation*}
Thanks to Young's inequality with $\eps>0$, each term of the right-hand side is estimated as follows (recall that $\gamma<1$ while $r,\zeta<p-1$):
\begin{equation*}
\begin{split}
\int_\Omega u_v^\alpha\dx & \leq\frac{\alpha}{p}\eps\Vert u_v\Vert_p^p
+\frac{p-\alpha}{p}C_\eps(\alpha)|\Omega|,\quad\alpha\in\{1-\gamma,r+1,1\};\\
\int_\Omega |D^sv|^\zeta u_v\dx & \leq\frac{1}{p}\eps\Vert u_v\Vert_p^p
+\frac{1}{p'}C_\eps\int_\Omega |D^s v|^{\zeta p'}\dx.
\end{split}
\end{equation*}
Consequently, by Proposition \ref{fracemb} (b) and \eqref{comparisonbound}, 
\begin{equation*}
\begin{split}
\|u_v\|^p_{s_1,p} & \le c\eps\|u_v\|_p^p+ C_\eps(1+\|D^s v\|^{\zeta p'}_{\zeta p'})
\le c'\eps\|u_v\|_{s_1,p}^p+ C'_\eps(1+\|D^s v\|^{\zeta p'}_p) \\
& \le c'\eps \|u_v\|_{s_1,p}^p +C'_\eps\left(1+\|v\|_{X^{s,p}_0(\R^N)}^{\zeta p'}\right)
\leq c'\eps \|u_v\|_{s_1,p}^p +C^*_\eps(1+\|v\|_{s_1,p}^{\zeta p'}). 
\end{split}
\end{equation*}
This entails
\begin{equation*}
 (1-c'\eps)\|u_v\|^p_{s_1,p}\le C^*_\eps(1+\|v\|_{s_1,p}^{\zeta p'}),
\end{equation*}
i.e., after choosing $\eps<\frac{1}{c'}$,
\begin{equation}\label{boundedness}
\| T(v)\|^{p}_{s_1,p}=\|u_v\|^{p}_{s_1,p}\le\hat{C}(1+ \|v\|^{\zeta p'}_{s_1,p}),
\end{equation}
with $\hat{C}:=\frac{C^*_\eps}{1-c'\eps}$. Since $\zeta p'<p$, there exists $\rho>0$ such that $\hat{C}(1+\rho^{\zeta p'})\le\rho^p$. Thus, due to \eqref{boundedness}, $\|v\|_{s_1,p}\le\rho$ implies $\| T(v)\|_{s_1,p}\le\rho$, which clearly means $T(K)\subseteq K$, provided
$$K:=\{u\in W^{s_1,p}_0(\Omega):\|u\|_{s_1,p}\le \rho\}.$$
\vskip1pt
\noindent {\it Claim 1:} The operator $T\lfloor_K$ is compact. 
\vskip1pt
\noindent Let $\{v_n\} \subseteq K$ and let $u_n:=T(v_n)$, $n\in\N$. The reflexivity of $W^{s_1,p}_0(\Omega)$ yields $v_n\rightharpoonup v$ in $W^{s_1,p}_0(\Omega)$ while Proposition \ref{fracemb} (c) ensures that
\begin{equation*}
\forall\, r\in [1, p_{s_1*})\;\;\text{one has}\;\; v_n\to v\;\;\text{in}\;\; L^r(\Omega),
\end{equation*}
where a sub-sequence is considered when necessary. Likewise, from $\{u_n\}\subseteq K$ it follows $u_n\rightharpoonup u$ in $W^{s_1,p}_0(\Omega)$ and, as before,  
\begin{equation}\label{u_n}
\forall\, r\in [1, p_{s_1*})\;\;\text{one has}\;\; u_n\to u\;\;\text{in}\;\; L^r(\Omega).
\end{equation} 
Now, testing \eqref{auxprob} with $\phi_n:=u_n-u$ and using \eqref{hypf}, Theorem \ref{existence}, and \eqref{hopf}, we obtain
\begin{equation*}
\begin{split}
& \langle(-\Delta)^{s_1}_p u_n+(-\Delta)^{s_2}_q u_n,\phi_n\rangle
\le\int_\Omega f(\cdot,u_n)|\phi_n|\dx+\int_\Omega g(\cdot,D^s v_n)|\phi_n|\dx\\
& \le c_1\int_\Omega u_n^{-\gamma}|\phi_n|\dx
+c_2\int_\Omega u_n^r|\phi_n|\dx +c_3\int_\Omega (1+|D^s v_n|^\zeta)|\phi_n|\dx\\
& \le \int_\Omega[c_1(\eta d^{s_1})^{-\gamma}+c_3]|\phi_n|\dx
+c_2\int_\Omega u_n^r|\phi_n|\dx +c_3\int_\Omega|D^s v_n|^\zeta |\phi_n|\dx.
\end{split}    
\end{equation*}
Hence, due to H\"older's inequality, Proposition \ref{distsumm} (recall that $s_1\gamma<\frac{1}{p'}$),  Proposition \ref{fracemb} (b), besides \eqref{comparisonbound},
\begin{equation}\label{diff}
\begin{split}
\langle(-\Delta)^{s_1}_p & u_n+(-\Delta)^{s_2}_q u_n,\phi_n\rangle\\
& \le C_1\|\phi_n\|_p
+ c_2\|u_n\|^r_p\|\phi_n\|_{\frac{p}{p-r}}
+c_3\|D^s v_n\|^\zeta_p\|\phi_n\|_{\frac{p}{p-\zeta}}\\
&\le C_1\|\phi_n\|_p+C_2\rho^r\|\phi_n\|_{\frac{p}{p-r}}
+C_3\rho^\zeta \|\phi_n\|_{\frac{p}{p-\zeta}}
\end{split}    
\end{equation}
for all $n\in\N$, where
$$C_1:=c_1\eta^{-\gamma}\Vert d^{-s_1\gamma}\Vert_{p'} +c_3|\Omega|^{1/p'}.$$ 
On account of \eqref{u_n}--\eqref{diff} one arrives at
\begin{equation*}
\limsup_{n\to+\infty}\langle(-\Delta)^{s_1}_p u_n+(-\Delta)^{s_2}_q u_n,u_n-u\rangle \le 0,   
\end{equation*}
whence $u_n\to u$ in $W^{s_1,p}_0(\Omega)$ because, by Proposition \ref{sumop} and $({\rm p}_4)$, the fractional $(p,q)$-Laplacian
$$u\mapsto (-\Delta)^{s_1}_p u_+(-\Delta)^{s_2}_q u,\quad u\in W^{s_1,p}_0(\Omega),$$
is type $({\rm S})_+$.
\vskip3pt
\noindent {\it Claim 2:} The operator $T\lfloor_K$ turns out continuous.
\vskip1pt
\noindent
Let $\{v_n\}\subseteq K$ satisfy $v_n\to v$ in $W^{s_1,p}_0(\Omega)$ and let $u_n:=T(v_n)$, $n\in\N$. Since $T\lfloor_K$ is compact, along a sub-sequence if necessary, we have $u_n\to u$ in $W^{s_1,p}_0(\Omega)$. Moreover, \eqref{u_n} holds. Our claim thus becomes $u=T(v)$. Pick any $\varphi\in W^{s_1,p}_0(\Omega)$. From \eqref{definitionT} it follows
\begin{equation}\label{u_nsolution}
\begin{split}
\int_{\R^{2N}} &
\frac{|u_n(x)-u_n(y)|^{p-2}(u_n(x)-u_n(y))(\phi(x)-\phi(y))}{|x-y|^{N+s_1p}}\dx\dy\\
& \hskip1cm+\int_{\R^{2N}}
\frac{|u_n(x)-u_n(y)|^{q-2}(u_n(x)-u_n(y))(\phi(x)-\phi(y))}{|x-y|^{N+s_2q}}\dx\dy\\
& =\int_\Omega f(\cdot,u_n)\phi\dx+\int_\Omega g(\cdot,D^s v_n)\phi\dx\quad\forall\, n\in\N.
\end{split}    
\end{equation}
Observe that
$$\{ u_n\}\subseteq K\implies 
\left\{\frac{|u_n(x)-u_n(y)|^{p-2}(u_n(x)-u_n(y))}{|x-y|^{\frac{N+s_1p}{p'}}}  \right\}\;\;\mbox{bounded in}\;\; L^{p'}(\R^{2N})$$
and that, by \eqref{u_n},
$$\lim_{n\to+\infty}
\frac{|u_n(x)-u_n(y)|^{p-2}(u_n(x)-u_n(y))}{|x-y|^{\frac{N+s_1p}{p'}}}
=\frac{|u(x)-u(y)|^{p-2}(u(x)-u(y))}{|x-y|^{\frac{N+s_1p}{p'}}}$$
for almost every $(x,y)\in\R^{2N}$. So, up to sub-sequences, 
\begin{equation*}
\frac{|u_n(x)-u_n(y)|^{p-2}(u_n(x)-u_n(y))}{|x-y|^{|x-y|^{\frac{N+s_1 p}{p'}}}}
\rightharpoonup\frac{|u(x)-u(y)|^{p-2}(u(x)-u(y))}{|x-y|^{\frac{N+s_1 p}{p'}}}
\;\;\mbox{in}\;\; L^{p'}(\R^{2N}). 
\end{equation*}
This implies
\begin{equation}\label{convergencelhs1}
\begin{split}
& \lim_{n\to+\infty}\int_{\R^{2N}}
\frac{|u_n(x)-u_n(y)|^{p-2}(u_n(x)-u_n(y))(\phi(x)-\phi(y))}{|x-y|^{N+s_1p}}\dx\dy\\
& \hskip3cm =\int_{\R^{2N}}
\frac{|u(x)-u(y)|^{p-2}(u(x)-u(y))(\phi(x)-\phi(y))}{|x-y|^{N+s_1p}}\dx\dy,
\end{split}
\end{equation}
because
$$\frac{\phi(x)-\phi(y)}{|x-y|^{\frac{N+s_1p}{p}}}\in L^p(\R^{2N}).$$ 
An analogous argument, which employs the continuous embedding $W^{s_1,p}_0(\Omega) \hookrightarrow W^{s_2,q}_0(\Omega)$ (cf. Proposition \ref{fracemb} (a)),  produces 
\begin{equation}\label{convergencelhs2}
\begin{split}
& \lim_{n\to+\infty}\int_{\R^{2N}}
\frac{|u_n(x)-u_n(y)|^{q-2}(u_n(x)-u_n(y))(\phi(x)-\phi(y))}{|x-y|^{N+s_2q}}\dx\dy\\
& \hskip3cm =\int_{\R^{2N}}
\frac{|u(x)-u(y)|^{q-2}(u(x)-u(y))(\phi(x)-\phi(y))}{|x-y|^{N+s_2q}}\dx\dy.
\end{split}
\end{equation}
Let us next focus on the right-hand side of \eqref{u_nsolution}. Exploiting \eqref{hypf}, Theorem \ref{existence}, \eqref{hopf}, \eqref{u_n}, and \cite[Theorem 4.9] {Br},  we achieve
$$|f(\cdot,u_n)\phi|\le [c_1 u_n^{-\gamma}+c_2 u_n^{r}]|\phi|
\le [c_1\underline{u}^{-\gamma}+c_2 u_n^r]|\phi|
\le [c_1(\eta d^{s_1})^{-\gamma}+c_2\psi^r]|\phi|,\quad n\in\N,$$
for some $\psi\in L^r(\Omega)$. 
Now, by \eqref{u_n} and \cite[Theorem 4.2]{Br}, one has
\begin{equation}\label{convergencerhs1}
\lim_{n\to+\infty}\int_\Omega f(\cdot,u_n)\phi\dx=\int_\Omega f(\cdot,u)\phi\dx.
\end{equation}
Finally, observe that, thanks to \eqref{comparisonbound},
\begin{equation*}
v_n\to v\;\mbox{in}\; W^{s_1,p}_0(\Omega)\implies v_n\to v\;\mbox{in}\; L^{s,p}_0(\Omega)
\end{equation*}
as well as
\begin{equation*}
v_n\to v\;\mbox{in}\; L^{s,p}_0(\Omega)\implies D^s v_n\to D^s v\;\mbox{in}\; L^p(\Omega)\implies (D^s v_n)^\zeta\to (D^s v)^\zeta\;\mbox{in}\; 
L^\frac{p}{\zeta}(\Omega),\end{equation*}
where a sub-sequence is considered if necessary. Since \eqref{hypg} holds while $\phi\in L^{\frac{p}{p-\zeta}}(\Omega)$ because $\zeta\in(1, p-1)$, this forces
\begin{equation}\label{convergencerhs2}
\lim_{n\to+\infty}\int_\Omega g(\cdot,D^s v_n)\phi\dx=
\int_\Omega g(\cdot,D^s v)\phi\dx.
\end{equation}
Letting $n\to+\infty$ in \eqref{u_nsolution} and using \eqref{convergencelhs1}--\eqref{convergencerhs2}, we arrive at $u=T(v)$.
\vskip3pt
Now, Schauder’s fixed point theorem can be applied to $T\lfloor_K$, which ends the proof.
\end{proof}
Our main result is the following:
\begin{thm}
Under hypotheses \eqref{hypf},\eqref{hypf2}, \eqref{hypg} and the conditions $q's_2\neq s_1< \frac{1}{p'\gamma}$, problem \eqref{prob} admits a weak solution $u\in W^{s_1,p}_0 (\Omega)$.
\end{thm}
\begin{proof}
Simply use Lemma \ref{Schau} and note that fixed points of $T$ weakly solve \eqref{prob}.
\end{proof}
\section*{Acknowledgments}
\noindent
This study was partly funded by: Research project of MIUR (Italian Ministry of Education, University and Research) Prin 2022 {\it Nonlinear differential problems with applications to real phenomena} (Grant No. 2022ZXZTN2).

The authors are members of the {\em Gruppo Nazionale per l'Analisi Matematica, la Probabilit\`a e le loro Applicazioni}
(GNAMPA) of the {\em Istituto Nazionale di Alta Matematica} (INdAM).


\begin{thebibliography}{99}
\bibitem{BC}
J.C. Bellido, J. Cueto, and C. Mora-Corral, \textit{Non-local gradients in bounded domains motivated by continuum mechanics: Fundamental theorem of calculus and embeddings}, Adv. Nonlinear Anal. \textbf{12} (2023), paper no. 20220316.
%
\bibitem{BF}
L. Brasco and G. Franzina, \textit{Convexity properties of Dirichlet integrals and Picone-type inequalities}, Kodai Math. J. {\bf 37} (2014), 769--799.
%
\bibitem{Br}
H. Brézis, \textit{Functional analysis, Sobolev spaces and partial differential equations}, Universitext, Springer, New York, 2011.
%
\bibitem{B}
E. Bruè, M. Calzi, G.E. Comi, and G. Stefani, \textit{A distributional approach to fractional Sobolev spaces and fractional variation: asymptotics II}, C. R. Math. Acad. Sci. Paris \textbf{360} (2022), 589--626. 
%
\bibitem{CSM}
A. Canino, L. Montoro, B. Sciunzi, and M. Squassina, \textit{Nonlocal problems with singular nonlinearity}, Bull. Sci. Math. \textbf{141} (2017), 223--250.
%
\bibitem{CS1}
G.E. Comi and G. Stefani, \textit{A distributional approach to fractional Sobolev spaces and fractional variation: Existence of blow-up}, J. Funct. Anal. \textbf{277} (2019), 3373--3435.
%
\bibitem{CS2}
G.E. Comi and G. Stefani, \textit{A distributional approach to fractional Sobolev spaces and fractional variation: Asymptotics I}, Rev. Mat. Complut. \textbf{36} (2023), 491--569.
%
\bibitem{DNPV}
E. Di Nezza, G. Palatucci, and E. Valdinoci, \textit{Hitchhiker's guide to the fractional Sobolev spaces}, Bull. Sci. Math. \textbf{136} (2012), 521--573.
%
\bibitem{Dy}
B. Dyda, \textit{A fractional order Hardy inequality}, Illinois J. Math. \textbf{48} (2004), 575--588.
%
\bibitem{FI}
S. Frassu and A. Iannizzotto, \textit{Extremal constant sign solutions and nodal solutions for the fractional $p$-Laplacian}. J. Math. Anal. Appl. \textbf{501} (2021), paper no. 124205.
%
\bibitem{GMM2}
L. Gambera, S.A. Marano, and D. Motreanu, \textit{Dirichlet problems with fractional competing operators and fractional convection}, Fract. Calc. Appl. Anal. \textbf{27} (2024), 2203--2218.
%
\bibitem{G}
J. Giacomoni, D. Kumar, and K. Sreenadh, \textit{Interior and boundary regularity results for strongly nonhomogeneous $p,q$-fractional problems}, Adv. Calc. Var. \textbf{16} (2023), 467--501.
%
\bibitem{GM}
U. Guarnotta and S.A. Marano, \textit{Strong solutions to singular discontinuous $p$-Laplacian problems}, arXiv:2407.20971 [math.AP]
%
\bibitem{H}
J. Horváth, \textit{On some composition formulas}, Proc. Amer. Math. Soc. \textbf{10} (1959), 433--437.
%
\bibitem{IMS}
A. Iannizzotto, S. Mosconi, and M. Squassina, \textit{Global Hölder regularity for the fractional p-Laplacian}, Rev. Mat. Iberoam. \textbf{32} (2016), 1353–-1392.
%
\bibitem{IMS2}
A. Iannizzotto, S.J.N. Mosconi, and M. Squassina, \textit{Fine boundary regularity for the degenerate fractional p-Laplacian}, J. Funct. Anal. \textbf{279} (2020), paper no. 108659.
%
\bibitem{IPS}
A. Iannizzotto, S. Liu, K. Perera, and M. Squassina, \textit{Existence results for fractional $p$-Laplacian problems via Morse theory}, Adv. Calc. Var. \textbf{9} (2016) 10--125.
%
\bibitem{LL}
E. Lindgren and P. Lindqvist, \textit{Fractional eigenvalues}, Calc. Var. Partial Differential Equations \textbf{49} (2014) 795--826.
%
\bibitem{MM}
S.A. Marano and S.J.N. Mosconi, \textit{Asymptotics for optimizers of the fractional Hardy-Sobolev inequality} Commun. Contemp. Math.\textbf{21} (2019), paper no. 18500028.
%
\bibitem{MS} 
P. Mironescu and W. Sickel, \textit{A Sobolev non embedding}, Atti Accad. Naz. Lincei Cl. Sci. Fis. Mat. Natur. Rend. Lincei (9) Mat. Appl. \textbf{26} (2015), 291--298.
%
\bibitem{MI}
Y. Mizuta, \textit{Potential Theory in Euclidean Spaces}, Gakkotosho, Tokyo, 1996.
%
\bibitem{SSS}
A. Schikorra, T.T. Shieh, and D. Spector, \textit{$L^p$-theory for fractional gradient PDE with VMO coefficients}, Atti Accad. Naz. Lincei Cl. Sci. Fis. Mat. Natur. Rend. Lincei (9) Mat. Appl. \textbf{26} (2015), 433--443.
%
\bibitem{SS1}
T.T. Shieh and D. E. Spector, \textit{ On a new class of fractional partial differential equations}, Adv. Calc. Var. \textbf{8} (2015), 321--336.
%
\bibitem{SS2}
T.T. Shieh and D.E. Spector, \textit{ On a new class of fractional partial differential equations II}, Adv. Calc. Var. \textbf{11} (2018), 289--307.
%
\bibitem{SI}
M. \u{S}ilhav\'{y}, \textit{Fractional vector analysis based on invariance requirements (critique of coordinate approaches)}, Contin. Mech. Thermodyn. \textbf{32} (2019), 207--228.
%
\end{thebibliography}
\end{document}